\documentclass[11pt]{article}
\usepackage{color}
\usepackage{cite}
\usepackage[english]{babel}
\usepackage[utf8]{inputenc}
\usepackage{amsmath}
\usepackage{amssymb}
\usepackage{amsthm}
\usepackage{amscd}
\usepackage{amsfonts}
\usepackage{newlfont}
\usepackage{enumerate}
\usepackage{fullpage}
\usepackage{graphicx}

\newtheorem{thm}{Theorem}[section]

\newtheorem{lem}[thm]{Lemma}
\newtheorem{prop}[thm]{Proposition}
\newtheorem{corol}[thm]{Corollary}

\theoremstyle{definition}
\newtheorem{definition}[thm]{Definition}
\newtheorem{rem}[thm]{Remark}
\newtheorem{example}[thm]{Example}

\date{}

\begin{document}

\title{Characterization of Filippov representable maps \\and Clarke subdifferentials}
\author{Mira BIVAS
\and Aris DANIILIDIS
\and Marc QUINCAMPOIX }
\maketitle

\begin{abstract}
The ordinary differential equation $\dot{x}(t)=f(x(t)), \; t \geq 0 $, for $f$ measurable, is not sufficiently regular to guarantee existence of solutions. To remedy this we may
relax the problem by replacing the function $f$ with its Filippov
regularization $F_{f}$ and consider the differential inclusion $\dot{x}(t)\in
F_{f}(x(t))$ which always has a solution. It is interesting to know,
inversely, when a set-valued map $\Phi$ can be obtained as the Filippov
regularization of a (single-valued, measurable) function. In this work we give
a full characterization of such set-valued maps, hereby called Filippov
representable. This characterization also yields an elegant description of
those maps that are Clarke subdifferentials of a Lipschitz function.

\end{abstract}

\medskip

\textbf{Keywords}: Filippov regularization, Krasovskii regularization,
Differential inclusion,\\ cusco map, Clarke subdifferential.

\bigskip

\textbf{AMS Classification}: 34A60, 28A10, 28A20, 26E25, 58C06

\section{Introduction}

We consider the differential equation
\begin{equation}
\dot{x}(s)=f(x(s)),\ s\geq0\,,\ x(0)=x_{0}\,, \label{eq_difeq}%
\end{equation}
where $f:\mathbb{R}^{d}\longrightarrow\mathbb{R}^{d}$  is a bounded measurable
function and $ x_0 \in \mathbb{R}^{d}$  . The above Cauchy problem might have no solution due to the lack of
regularity of $f$. A way to overcome this difficulty is to replace
\eqref{eq_difeq} by a "minimal" differential inclusion which is sufficiently
regular to have a solution. A natural way to do this is to replace $f$ by
its Krasovskii regularization $K_{f}$ given by
\[
K_{f}(x):=\bigcap_{\delta>0}\overline{\text{co}}\,f(B_{\delta}(x))
\]
and obtain, accordingly:
\begin{equation}
\dot{x}(s)\in K_{f}(x(s)),\ x(0)=x_{0},\ s\geq0. \label{eq_krassol}%
\end{equation}
Another possibility is to consider, instead of $K_{f}$, the Filippov
regularization $F_{f}$ of $f$ given by%
\[
F_{f}(x):=\bigcap_{\mathcal{L}(N)=0}\bigcap_{\delta>0}\,\overline{\text{co}}\,f((B_{\delta}(x))\setminus N),
\]
where the first intersection is taken over the sets $N\subset\mathbb{R}^{d}$
with Lebesgue measure $\mathcal{L}(N)$ equal to zero. In this way, we obtain
the so-called Filippov solutions of \eqref{eq_difeq}, that is, solutions of
the differential inclusion
\begin{equation}
\dot{x}(s)\in F_{f}(x(s)),\ x(0)=x_{0},\ s\geq0. \label{eq_filsol}
\end{equation}
The Filippov regularization is based on the idea that sets of measure zero
should play no role in the relaxed dynamics.\smallskip

Inclusions \eqref{eq_krassol} and \eqref{eq_filsol} always have a solution,
since the set-valued mappings $K_{f}$ and $F_{f}$ are upper semicontinuous,
with nonempty convex compact values (\emph{c.f.} \cite{AubinCellina},
 \cite{Deimling}). For simplicity, borrowing terminology from \cite{BMW1997}, \cite{BMS1999}, we shall refer to such set-valued mappings as
\emph{cusco} maps (see forthcoming Definition~\ref{Def-cusco}). If the
function $f$ is continuous, then both maps $K_{f}$ and $F_{f}$ are
single-valued and equal to $f$.\smallskip

The techniques of Krasovskii and Filippov regularizations were introduced for
obtaining solutions of discontinuous differential equations. Both
regularizations have further been widely used in optimal control and
differential games, see \cite{Bac}, \cite{Cor}, \cite{Fil1960},
\cite{Hajek}, \cite{KrasSub}, \cite{Utkin}, \cite{Sontag} \emph{e.g.}
\smallskip

The main goal of this paper is to consider the inverse problem: given a cusco
set-valued mapping $F$ from $\mathbb{R}^{d}$ to $\mathbb{R}^{d}$, does there
exist a singe-valued function $f$, such that $F$ is the Krasovskii / Filippov
regularization of $f$? We shall refer to such maps
as Krasovskii representable (respectively, Filippov representable). Notice that "being cusco" is clearly a necessary
condition for being representable. We completely characterize Filippov representable maps, even in a slightly more general setting, namely, for maps defined in $\mathbb{R}^{d}$ with values in $\mathbb{R}^{\ell}$.

The other main contribution of this work is an equivalent characterization of the set-valued maps that are Clarke subdifferentials of a Lipschitz function in the finite-dimensional case. We show that these maps are exactly the Filippov regularizations of functions satisfying a so-called \emph{nonsmooth Poincar\'{e} condition}. This condition is recently stated and used independently in \cite{Flores-Master} and \cite{CK2017} for a different purpose. We refer to \cite{BMS1999} for another characterization of set-valued maps that are Clarke subdifferentials of a Lipschitz function in Banach spaces.

\smallskip

The manuscript is organized as follows: In Section~\ref{sec:2} we introduce
basic notation and background for Krasovskii and Filippov regularizations. In
Section~\ref{sec:3} we obtain several key results for both regularizations,
while in Section~\ref{sec:4} we provide the main result (characterization of
Filippov representability) and use it to obtain an alternative
characterization of those set-valued maps that are Clarke subdifferentials of
Lipschitz functions (Section~5).

\section{Preliminaries}

\label{sec:2}Throughout the paper, we denote by $B_{X}$ (respectively,
$\bar{B}_{X}$) the open (respectively, closed) unit ball, centered at the
origin of the normed space $X$. The index will often be omitted if there is no
ambiguity about the space. In this case, we denote by $B_{\delta}(x):=x+\delta
B_{X}$ the (open) ball centered at $x$ with radius $\delta.$ We also denote by
$\mathcal{L}_{d}$ the Lebesgue measure in $\mathbb{R}^{d}$ and by
$\mathcal{N}_{d}$ the set of $\mathcal{L}_{d}$-null subsets of $\mathbb{R}%
^{d},$ that is,
\[
\mathcal{N}_{d}=\{N\subset\mathbb{R}^{d}:\,\mathcal{L}_{d}(N)=0\}.
\]
We shall also omit the index $d$ and simply write $\mathcal{L}$ for the
Lebesgue measure and $\mathcal{N}$ for the family of null sets, whenever there
is no ambiguity about the dimension.\smallskip

For a set-valued mapping $\Phi$ from $\mathbb{R}^{d}$ to the subsets of
$\mathbb{R}^{\ell}$, we will use the notation $\Phi:\mathbb{R}^{d}%
\rightrightarrows\mathbb{R}^{\ell}$, while a (single-valued) function will be
denoted by $f:\mathbb{R}^{d}\longrightarrow\mathbb{R}^{\ell}$. The following
definition provides a convenient abbreviation for several statements in the sequel.

\begin{definition}
[Cusco map]\label{Def-cusco}An upper semi-continuous set-valued map
$\Phi:\mathbb{R}^{d}\rightrightarrows\mathbb{R}^{\ell}$ with nonempty compact
convex values will be called cusco.
\end{definition}

Under the above terminology, the Krasovskii regularization $K_{f}$ is the
smallest cusco map $\Phi$ satisfying $f(x)\in\Phi(x)$ for {all} $x\in
\mathbb{R}^{d}$ and the Filippov regularization $F_{f}$ is the smallest cusco map $\Psi$ satisfying $f(x)\in\Psi(x)$ for {\it almost all} $x\in \mathbb{R}^{d}$. We refer the reader to \cite{Fil1960}, \cite{Fil1988} and \cite{BOQ2009} for more information on Filippov's regularization and its applications. We also refer to \cite{BMS1999}, \cite{BMW1997} for properties of cusco maps.\smallskip

We shall also need the following classical notion of a point of approximate
continuity of a measurable function.

\begin{definition}
[Points of approximate continuity]\label{Def-approx-cont}Let $f:\mathbb{R}^{d}\rightarrow\mathbb{R}^{\ell}$ be a measurable function. A point $x\in\mathbb{R}^{d}$ is called a point of approximate continuity for $f$ if for every $\varepsilon>0$ it holds:
\begin{equation}
\lim_{\delta\rightarrow0^{+}}\frac{\mathcal{L}\{x^{\prime}\in B_{\delta
}(x),\;|f(x^{\prime})-f(x)|\geq\varepsilon\}}{\mathcal{L}(B_{\delta}(x))}=0.
\label{ac}
\end{equation}

\end{definition}

It is well-known that the complement $\mathbf{N}_{f}$ of the set of points of
approximate continuity of a locally bounded measurable $f:\mathbb{R}%
^{d}\rightarrow\mathbb{R}^{\ell}$ is $\mathcal{L}_{d}$-null (\emph{c.f.}
\cite{EG} e.g.). Based on this result we can establish the following useful lemma.

\begin{lem}
\label{lem1} Let $f:\mathbb{R}^{d}\rightarrow\mathbb{R}^{\ell}$ be a (locally)
bounded measurable function and $\mathbb{R}^{d}\diagdown\mathbf{N}_{f}$ be the
set of points of approximate continuity. Then for every $\bar{x}\in
\mathbb{R}^{d}$, $\delta>0$ and $N\in\mathcal{N}$ we have:
\begin{equation}
f(B_{\delta}(\bar{x})\setminus\mathbf{N}_{f})\subset\overline
{f(B_{\delta}(\bar{x})\setminus(\mathbf{N}_{f}\cup N))}
\quad\text{and\quad}\overline{\mathrm{co}}\,f(B_{\delta
}(\bar{x})\!\setminus\!\mathbf{N}_{f})=\overline{\mathrm{co}}
\,\left(f(B_{\delta}(\bar{x})\setminus(\mathbf{N}_{f}\cup
N))\right). \label{eq:L1-i}
\end{equation}

Consequently, for every $\bar{x}\in\mathbb{R}^{d}$ and $\delta>0$ it holds:
\begin{equation}
\overline{\mathrm{co}}\mathrm{\,\,}f(B_{\delta}(\bar{x}
)\mathbf{\diagdown}\mathbf{N}_{f})=
\bigcap_{N\in\mathcal{N}}
\overline{\mathrm{co}}\mathrm{\,\,}f(B_{\delta}(\bar{x})
\mathbf{\diagdown}N). \label{eq:L1-ii}
\end{equation}

\end{lem}

\begin{proof} Let us prove \eqref{eq:L1-i}. Fix $\varepsilon>0$,
$N\in\mathcal{N}$ and $x\in B_{\delta}(\bar{x})\diagdown\mathrm{\,\,}N_{f}$.
Take $\delta_{1}<\delta$ such that $B_{\delta_{1}}(x)\subset B_{\delta}
(\bar{x})$. By \eqref{ac}, there exists $\delta_{2}\in(0,\delta_{1})$ such
that
\[
\frac{\mathcal{L}\{x^{\prime}\in B_{\delta_{2}}(x),\;|f(x^{\prime})-f(x)|\geq\varepsilon\}}{\mathcal{L}(B_{\delta_{2}}(x))}<1,
\]
which yields
\[
\mathcal{L}\{x^{\prime}\in B_{\delta_{2}}(x),\;|f(x^{\prime})-f(x)|<\varepsilon\}>0.
\]
Thus
\[
\mathcal{L}(\{x^{\prime}\in B_{\delta_{2}}(x),\;|f(x^{\prime}
)-f(x)|<\varepsilon\}\diagdown\mathrm{\,\,}(N_{f}\cup N))>0.
\]
Hence there exists $x^{\prime}\in B_{\delta_{2}}(x)\diagdown\mathrm{\,\,}
(N_{f}\cup N))\subset B_{\delta}(\bar{x})\diagdown\mathrm{\,\,}(N_{f}\cup N))$
such that $|f(x^{\prime})-f(x)|<\varepsilon$. Since $\varepsilon$ is arbitrary
we deduce
\[
f(x)\in\overline{f(B_{\delta}(\bar{x})\diagdown\mathrm{\,\,}(N_{f}\cup N))}.
\]
The right-hand side of \eqref{eq:L1-i} follows from the fact that for every
subset $A$ of $\mathbb{R}^{\ell}$ we have
\[
A\subset\mathrm{co}(A)\implies\overline{A}\subset\overline{\text{co}
}(A)\implies\overline{\text{co}}(\overline{A})=\overline{\text{co}}(A).
\]
Assertion \eqref{eq:L1-ii} follows directly from \eqref{eq:L1-i}.
\end{proof}

We recall the following result due to Castaing (see \cite[Theorem~8.1.4]{AubinFrankowska91} e.g.)

\begin{prop}
\label{selec}Let $\Phi:\mathbb{R}^{d}\rightrightarrows\mathbb{R}^{\ell}$ be a
measurable set-valued map. Then there exists a sequence of measurable
selections $\{f_{n}\}_{n=1}^{\infty}$ of $\Phi$ such that
\[
\Phi(x)=\overline{\{f_{n}(x)\ |\ n\in\mathbb{N}\}},\quad\text{for all }
x\in\mathbb{R}^{d}.
\]

\end{prop}

Combining above proposition with Lemma~\ref{lem1}, we deduce the following
useful result.

\begin{corol} \label{corFi} Let $\Phi:\mathbb{R}^{d}\rightrightarrows\mathbb{R}^{\ell}$ be
cusco. Then there exists $\mathbf{N}_{\Phi}\in\mathcal{N}_{d}$ (Lebesgue null
set) such that for every $\bar{x}\in\mathbb{R}^{d}$, $\delta>0$ and
$N\in\mathcal{N}$ we have:
\begin{equation}
\Phi(B_{\delta}(\bar{x})\!\setminus\!\mathbf{N_{\Phi}})\subset\overline{\Phi(B_{\delta}(\bar{x})\setminus(\mathbf{N}_{\Phi
	}\!\cup N))}\quad\text{and}\quad\overline{\mathrm{co}}\,\Phi(B_{\delta}(\bar{x})\!\setminus\!N_{\Phi})=\overline{\mathrm{co}}\,\left(\Phi(B_{\delta}
(\bar{x})\!\setminus\!(\mathbf{N}_{\Phi}\!\cup\! N)\right).
\label{eq:cor-i}
\end{equation}

Consequently, for every $\bar{x}\in\mathbb{R}^{d}$ and $\delta>0$ it holds:
\begin{equation}
\overline{\text{\textrm{co}}}\mathrm{\,\,}\Phi(B_{\delta}(\bar{x}
)\mathbf{\diagdown N}_{\Phi}\mathrm{\,\,})=
\bigcap_{N\in\mathcal{N}}
\overline{\text{\textrm{co}}}\mathrm{\,\,}\Phi(B_{\delta}(\bar{x}
)\mathbf{\diagdown}N). \label{eq:cor-ii}
\end{equation}

\end{corol}

\begin{proof} Let $\left\{f_{n}\right\}_{n\geq1}$ be a sequence of
measurable sets associated to $\Phi$ ({\it c.f.} Proposition~\ref{selec}).
We set $\mathbf{N}_{\Phi}:=\bigcup_{k\geq1}N_{k}$, where $N_{k}=\mathbf{N}_{f_{k}}$ is the complement of the set of points of approximate continuity of $f_{k}$. We obviously have that $\mathbf{N}_{\Phi}$ is a null set. Let us show that \eqref{eq:cor-i} holds.\smallskip\newline To this end, let $N\in
\mathcal{N}$, $\bar{x}\in\mathbb{R}^{d}$ and $\delta>0$. Fix $x\in B_{\delta
}(\bar{x})\diagdown\mathrm{\,\,}N_{\Phi}$ and take $\delta_{1}\in(0,1)$ such
that $B_{\delta_{1}}(x)\subset B_{\delta}(\bar{x})$. By Lemma~\ref{lem1} we
have for any $k\geq1$,
\begin{align*}
f_{k}(x)  &  \in f_{k}(B_{\delta_{1}}(x)\diagdown\mathrm{\,\,}N_{k}%
)\subset\overline{f(B_{\delta_{1}}(\bar{x})\diagdown\mathrm{\,\,}(N_{k}\cup
N_{\Phi}\cup N))}=\overline{f(B_{\delta_{1}}(\bar{x})\diagdown\mathrm{\,\,}
(N_{\Phi}\cup N))}\\
&  \subset\overline{\Phi(B_{\delta}(\bar{x})\diagdown\mathrm{\,\,}(N_{\Phi
}\cup N))}.
\end{align*}
So
\[
\Phi(x)=\overline{\{f_{k}(x),\;k\geq1\}}\subset\overline{\Phi(B_{\delta}
(\bar{x})\diagdown\mathrm{\,\,}(N_{\Phi}\cup N))}\, ,
\]
which established the left-hand side of \eqref{eq:cor-i}. The remaining
assertions are easily deduced in a similar manner as in Lemma~\ref{lem1}.
\end{proof}

Let us now recall (see \cite[Proposition 2]{BOQ2009} \emph{e.g.}) the
following useful results. In \cite{BOQ2009}, the results below have been
stated and proved for the case $\ell=d$. The proofs for the general case
($\ell$ arbitrary) are identical. In what follows, $\mathcal{N}$ will always
denote the class of Lebesgue null sets.

\begin{prop}
\label{prop_filreg} Let $f:\mathbb{R}^{d}\rightarrow\mathbb{R}^{\ell}$ be a
measurable and (locally) bounded function. Then,
\begin{itemize}
\item[$\mathbf{(i).}$] there exists a set $\mathbf{N}_{f}\in\mathcal{N}$ such that
\[
F_{f}(x):=\bigcap_{\delta>0}\overline{\text{\textrm{co}}}\mathrm{\,\,}
\,f((B_{\delta}(x))\setminus\mathbf{N}_{f}),\quad\text{for all }x\in
\mathbb{R}^{d}
\]
and \,\,$f(x)\in F_{f}(x)$\,\, for almost all $x\in\mathbb{R}^{d}$.

\item[$\mathbf{(ii).}$] $F_{f}$ is the smallest cusco map $\Phi$ such that $f(x)\in\Phi(x),$ for
almost all $x\in\mathbb{R}^{d}$.

\item[$\mathbf{(iii).}$] $F_{f}$ is single-valued if and only if there exists a continuous
function $g$ which coincides almost everywhere with $f$. In this case,
$F_{f}(x)=\{g(x)\}$ for almost all $x\in\mathbb{R}^{d}$.

\item[$\mathbf{(iv).}$] there exists a (necessarily measurable) function $\bar{f}$ which is
equal almost everywhere to $f$ and such that
\[
F_{f}(x):=\bigcap_{\delta>0}\overline{\text{\textrm{co}}}\mathrm{\,\,}
\,\bar{f}(B_{\delta}(x)),\quad\text{for all }x\in\mathbb{R}^{d}.
\]

\item[$\mathbf{(v).}$] if a function $\widetilde{f}$ coincides with $f$ for almost all
$x\in\mathbb{R}^{d}$, then
\[
F_{f}(x)=F_{\widetilde{f}}(x),\quad\text{for all }x\in\mathbb{R}^{d}.
\]

\item[(vi).] for all $x\in\mathbb{R}^{d}$
\[
F(x)=\bigcap_{\widetilde{f}=f\text{a.e.}}\bigcap_{\delta>0}\overline{\mathrm{co }}\,\widetilde{f}(B_{\delta}(x))\,,
\]
where the first intersection is taken over all functions $\widetilde{f}$ equal to $f$ almost everywhere.
\end{itemize}
\end{prop}

\section{Cusco maps and Filippov representability}
\label{sec:3}

Before we proceed, we shall need the following classical result, whose proof
is provided for completeness. According to the terminology of Kirk
\cite{Kirk}, the result asserts the existence, for every Euclidean space, of a
countable partition that splits the family of open sets. For alternative proofs, or proofs of similar statements see \cite{Wang-2004}, \cite{DF2019}, \cite{DD2019}.

\begin{lem}
[Splitting partition]\label{partLem}There exists a partition $\{A_{n}
\}_{n=1}^{\infty}$ of $\mathbb{R}^{d}$, such that for every $n\in\mathbb{N}$
the set $A_{n}$ has a positive measure in every open subset of $\mathbb{R}
^{d}$.
\end{lem}

\begin{proof} Consider the countable family $\mathcal{U}_{1}, \mathcal{U}
_{2}, \dots$ of open balls with rational centers and rational radii in
$\mathbb{R}^{d}$. Let
\[
b : \mathbb{N} \times\mathbb{N} \to\mathbb{N}
\]
be a bijection such that $b(1, 1) = 1$.

Using that each nonempty open set contains a closed nowhere dense set with
positive measure (\textit{e.g.} a Smith--Volterra--Cantor set, also called ``fat" Cantor set), we can choose $T_{1}\subset\mathcal{U}_{1}$ to be a nowhere
dense closed set with positive measure. Then, we construct a sequence
$\{T_{m}\}_{m=2}^{\infty}$ of disjoint closed nowhere dense sets with positive
measure such that
\begin{equation}
\text{if }m=b(k,j)\text{, then }T_{m}\subset \,\mathcal{U}_{k}\diagdown\cup_{l<m}T_{l}\,. \label{constr}
\end{equation}
This can be done since the set $\mathcal{U}_{k}\diagdown \bigcup_{l<m}T_{l}$ is open.

We now set
\[
A_{n}:=\bigcup_{k=1}^{\infty}T_{b(k,n)},\ n\geq2
\]
and
\[
A_{1}:=\mathbb{R}^{d}\setminus\bigcup_{n=2}^{\infty}A_{n}\,.
\]
It is clear that $\{A_{n}\}_{n=1}^{\infty}$ are measurable and disjoint.
Moreover, if $O$ be a nonempty open set, then there exists $k$ such that
$\mathcal{U}_{k}\subset O$. Using \eqref{constr}, we obtain that
\[
A_{n}\cap O\supset A_{n}\cap\mathcal{U}_{k}\supset T_{b(k,n)}\,,\ n\geq2
\]
and
\[
A_{1}\cap O\supset(\mathbb{R}^{d}\setminus\bigcup_{n=2}^{\infty}A_{n}
)\bigcap\mathcal{U}_{k}\supset T_{b(k,1)}\,.
\]
Hence, $\mathcal{L}(A_{n}\bigcap O)\geq\mathcal{L}(T_{b(k,n)})>0$ and
$\mathcal{L}(A_{1}\cap O)\geq\mathcal{L}(T_{b(k,1)})>0$. This completes the
proof of the lemma.

\end{proof}

We are now ready to prove the following result.

\begin{thm}
\label{filTh} Let $\Phi:\mathbb{R}^{d}\rightrightarrows\mathbb{R}^{\ell}$ be a
cusco map. Then there exists a measurable function $f:\mathbb{R}%
^{d}\rightarrow\mathbb{R}^{\ell}$ such that $\Phi$ is almost everywhere equal
to $F_{f}$ (the Filippov regularization of $f$), that is:
\[
\Phi(x)=F_{f}(x),\quad\text{for almost every }x\in\mathbb{R}^{d}.
\]

\end{thm}

\begin{proof} In view of Proposition~\ref{selec}, there exists a sequence of
measurable selections $\{f_{n}\}_{n=1}^{\infty}$ of $\Phi$ such that
\[
\Phi(x)=\overline{\{f_{n}(x)\ |\ n\in\mathbb{N}\}},\quad\text{for every }
x\in\mathbb{R}^{d}.
\]
Let $\{A_{n}\}_{n=1}^{\infty}$ be a splitting partition of $\mathbb{R}^{d}$
given in Lemma~\ref{partLem}. We define the measurable function $f:\mathbb{R}%
^{d}\rightarrow\mathbb{R}^{d}$ as follows:
\[
f(x):=\sum_{n=1}^{\infty}f_{n}(x)\mathbf{1}_{A_{n}}(x),
\]
where $\mathbf{1}_{A}$ denotes the characteristic function of the set $A$
(equal to $1$ if $x\in A$ and to $0$ if $x\notin A$). Let
\[
F_{f}(x):=\bigcap_{N,\mathcal{L}(N)=0}\bigcap_{\delta>0}\overline{\text{co}
}\,f(B_{\delta}(x)\setminus N)
\]
be the Filippov regularization of $f$. Since $\mathcal{L}(B_{\delta}(x)\cap
A_{n})>0$ for all $n\in\mathbb{N}$ and for all $\delta>0$, we obtain that
\begin{equation}
\begin{split}
F_{f}(x)  &  \supset\bigcap_{N,\,\mathcal{L}(N)=0}\bigcap_{\delta>0}
\overline{\text{co}}\,f((B_{\delta}(x)\cap A_{n})\setminus N)\\
&  =\bigcap_{N,\,\mathcal{L}(N)=0}\bigcap_{\delta>0}\overline{\text{co}
}\,f_{n}((B_{\delta}(x)\cap A_{n})\setminus N)
\end{split}
\label{ff}
\end{equation}
for all $n\in\mathbb{N}$, $x\in\mathbb{R}^{d}$.

The next step in the proof consists in showing that the last expression in
\eqref{ff} contains $f_{n}(x)$ for almost all $x\in\mathbb{R}^{d}$. In order
to do it, we will need the following assertion.\smallskip

\textbf{Claim.} There exists a sequence of measurable sets $\{K_{m}%
\}_{m=1}^{\infty}$ such that:

\begin{enumerate}
\item $K_{1}\subset K_{2} \subset\dots K_{m} \subset\dots$

\item $\mathbb{R}^{d}=\bigcup_{m=1}^{\infty}K_{m}\cup N_{0}$, where
$\mathcal{L}(N_{0})=0$

\item the restrictions $f_{n}\vert_{K_{m}}$ are continuous for all $m, n
\in\mathbb{N}$.
\end{enumerate}

We postpone the proof of the claim at the end of this proof. Assuming the
above claim, we deduce from Lemma \ref{partLem} that for all $n\in\mathbb{N}$,
$x\in\mathbb{R}^{d}$ and $\delta>0$ it holds:%
\begin{align*}
0<  &  \mathcal{L}(B_{\delta}(x)\cap A_{n})=\mathcal{L}(B_{\delta}(x)\cap
A_{n}\cap(\mathbb{R}^{d}\setminus N_{0}))\\
=  &  \mathcal{L}\left(  B_{\delta}(x)\cap A_{n}\cap\bigcup_{m=1}^{\infty
}K_{m}\right)  =\mathcal{L}\left(  \bigcup_{m=1}^{\infty}(B_{\delta}(x)\cap
A_{n}\cap K_{m})\right) \\
=  &  \lim_{m\rightarrow\infty}\mathcal{L}(B_{\delta}(x)\cap A_{n}\cap
K_{m})\,,
\end{align*}
since $K_{m}\subset K_{m+1}$ for all $m\in\mathbb{N}$. Therefore, for some
$m_{0}\in\mathbb{N}$ sufficiently large we have
\[
\mathcal{L}(B_{\delta}(x)\cap A_{n}\cap K_{m})>0,
\]
for all $n\in\mathbb{N}$, $x\in\mathbb{R}^{d}$, $\delta>0$ and $m\geq m_{0}.$

Let us fix an arbitrary $x\notin N_{0}$. Then, $x\in K_{m_{1}}$ for some
$m_{1}\in\mathbb{N}$. Let $\bar{m}:=\max(m_{0},m_{1})$. Since $x\in K_{m}$ for
all $m\geq m_{1}$, we can continue \eqref{ff} in the following way
\[
F_{f}(x)\supset\bigcap_{N,\,\mathcal{L}(N)=0}\bigcap_{\delta>0}\overline
{\text{co}}\,f_{n}(B_{\delta}(x)\cap A_{n}\cap K_{\bar{m}}\setminus N,t)\ni
f_{n}(x)\,,
\]
where the last inclusion is due to continuity of $f_{n}|_{K_{\bar{m}}}$.

We have obtained that for all $n\in\mathbb{N}$ and for all $x\in\mathbb{R}^{d}\setminus N_{0}$
\[
F_{f}(x)\ni f_{n}(x)\,.
\]
Since the Filippov regularization $F_{f}$ is closed-valued, we obtain
\[
F_{f}(x)\supset\Phi(x)\ni f(x),\qquad\text{for all }x\in\mathbb{R}
^{d}\setminus N_{0}.
\]

We deduce from Proposition~\ref{prop_filreg} $(ii)$ that $F_{f}(x)=\Phi(x)$
for almost every $x\in\mathbb{R}^{d}$.

It remains to prove the claim about the
existence of the sequence of sets $\{K_{m}\}_{m=1}^{\infty}$. Since the
functions $f_{n}$ are measurable, due to Lusin's theorem, for every
$m,n\in\mathbb{N}$ we can find a set $K_{n,m}\subset\mathbb{R}^{d}$ such that
$f_{n}|_{K_{n,m}}$ is continuous and
\[
\mathcal{L}(\mathbb{R}^{d}\setminus K_{n,m})<\frac{1}{2^{n+m}}\,.
\]

Let us set $K_{m}^{\prime}:=\bigcap_{n=1}^{\infty}K_{n,m}$. We have that the
restrictions $f_{n}|_{K_{m}^{\prime}}$ are continuous for all $m,n\in
\mathbb{N}$ and
\[
\mathcal{L}(\mathbb{R}^{d}\setminus K_{m}^{\prime})=\mathcal{L}\left(
\bigcup_{n=1}^{\infty}(\mathbb{R}^{d}\setminus K_{n,m})\right)  \leq\sum
_{n=1}^{\infty}\mathcal{L}(\mathbb{R}^{d}\setminus K_{n,m})<\sum_{n=1}^{\infty}\frac{1}{2^{n+m}}=\frac{1}{2^{m}}\,.
\]

The inclusions $K_{1}\subset K_{2}\subset\dots K_{m}\subset\dots$ are obtained
by taking
\[
K_{m}:=\bigcap_{l\geq m}K_{l}^{\prime},\ m=1,2,\dots\,.
\]
We have that
\[
\mathcal{L}(\mathbb{R}^{d}\setminus K_{m})=\mathcal{L}\left(  \bigcup
_{l=m}^{\infty}(\mathbb{R}^{d}\setminus K_{l}^{\prime})\right)  \leq\sum
_{l=m}^{\infty}\mathcal{L}(\mathbb{R}^{d}\setminus K_{l}^{\prime})<\sum
_{l=m}^{\infty}\frac{1}{2^{l}}=\frac{1}{2^{m-1}}\,.
\]
Let us set $N_{0}:=\mathbb{R}^{d}\setminus\cup_{m=1}^{\infty}K_{m}$. Since
$\mathbb{R}^{d}\setminus K_{m+1}\subset\mathbb{R}^{d}\setminus K_{m}$, we
obtain that
\[
\mathcal{L}(N_{0})=\mathcal{L}\left(  \bigcap_{m=1}^{\infty}(\mathbb{R}%
^{d}\setminus K_{m})\right)  =\lim_{m\rightarrow\infty}\frac{1}{2^{m-1}}=0\,.
\]
The proof is complete.
\end{proof}

We also obtain the following

\begin{prop}\label{th_main}Let $\Phi:\mathbb{R}^{d}\rightrightarrows
\mathbb{R}^{\ell}$ be a cusco map. Then, there exists a measurable selection
$f:\mathbb{R}^{d}\rightarrow\mathbb{R}^{\ell}$ of $\Phi$ (that is, $f(x)\in
\Phi(x)$ for all $x\in\mathbb{R}^{d}$), such that

\begin{enumerate}

\item[$\mathbf{(i).}$] $\Phi$ is equal almost everywhere to the Filippov regularization of
$f$, that is,
\[
\Phi(x)=F_{f}(x),\quad\text{for almost all }x\in\mathbb{R}^{d}.
\]

\item[$\mathbf{(ii).}$] there exists some $\hat{f}:\mathbb{R}^{d}\rightarrow\mathbb{R}^{\ell}$  such that $\Phi$ is equal almost everywhere to
the Krasovskii regularization of $\hat f$, that is,
\[
\Phi(x)=K_{\hat{f}}(x),\quad\text{\ for almost all }x\in\mathbb{R}^{d}.
\]

\item[$\mathbf{(iii).}$] $\Phi$ is equal almost everywhere to the intersection of all Filippov
regularizations defined by functions $\widetilde{f}$ which are equal to $f$
almost everywhere, that is,
\[
\Phi(x)=\bigcap_{\widetilde{f}=f\text{a.e.}}F_{\widetilde{f}}(x),\quad
\text{for almost all }x\in\mathbb{R}^{d}.
\]

\end{enumerate}
\end{prop}

\begin{proof} Using Theorem \ref{filTh}, we obtain a measurable function
$\bar{f}:\mathbb{R}^{d}\rightarrow\mathbb{R}^{\ell}$ such that $\Phi$ is equal
almost everywhere to the Filippov regularization $F_{f}$ of $\bar{f}$, that
is,
\[
\Phi(x)=\bigcap_{N,\mathcal{L}(N)=0}\bigcap_{\delta>0}\overline
{\text{co\thinspace}}\,\bar{f}(B_{\delta}(x)\setminus N),\quad\text{for almost
every }x\in\mathbb{R}^{d}.
\]

Due to Proposition~\ref{prop_filreg} $(iv)$ there exists a function $\hat
{f}:\mathbb{R}^{d}\rightarrow\mathbb{R}^{\ell}$ such that for all
$x\in\mathbb{R}^{d}$
\[
\Phi(x):=\bigcap_{\delta>0}\overline{\text{co}}\,\,\hat{f}\left(  B_{\delta
}(x)\right)  \,.
\]
Clearly at every point $x\in\mathbb{R}^{d}\diagdown\mathbf{\hat{N}}_{f}$ of
approximate continuity of $\hat{f}$ we have that $ \hat f(x)\in\Phi(x)$. So setting
$f(x)=\hat{f}(x),$ whenever $x\in\mathbb{R}^{d}\diagdown\mathbf{N}_{\hat{f}}$
and taking $f(x)$ to be any element of $\Phi(x)$ if $x\in\mathbf{N}_{\hat{f}}$,
we obtain both claims (i) and (ii).

In order to establish $(iii)$, we use $(i)$ to obtain that for all
$x\in\mathbb{R}^{d}\setminus\mathbf{N}_{\hat{f}}$
\[
\Phi(x)=\bigcap_{\delta>0}\overline{\text{co}}\,f\left(  B_{\delta}(x)\right)
\supset\bigcap_{\widetilde{f}=f\text{a.e.}}\bigcap_{\delta>0}\overline
{\text{co}}\,\widetilde{f}\left(  B_{\delta}(x)\right)  \,.
\]
At the same time we also have:
\[
\bigcap_{\widetilde{f}=f\text{a.e.}}\bigcap_{\delta>0}\overline{\text{co}
}\,f\left(  B_{\delta}(x)\right)  \supset\bigcap_{\widetilde{f}=f\text{a.e.}
}\bigcap_{N,\mathcal{L}(N)=0}\bigcap_{\delta>0}\overline{\text{co}
}\,\widetilde{f}(B_{\delta}(x)\setminus N)\,.
\]

\noindent The right-hand side is $\bigcap_{\widetilde{f}=f\text{a.e.}}F_{\widetilde{f}
}(x)$, which by Proposition~\ref{prop_filreg} $(vi)$ is equal to $F_f(x)$, for all $x\in\mathbb{R}^{d}$.\smallskip

\noindent The proof is complete.
\end{proof}

\begin{rem}
Notice that (completely) different functions may give rise to the same
Filippov regularization: Indeed, let $A\subset\mathbb{R}$ be a splitting set,
that is, $A$ and $\mathbb{R}\diagdown A$ have positive measure on every
nontrivial interval. Then both $f(x):=\mathbf{1}_{A}(x)$ and $\widetilde
{f}(x):=\mathbf{1}_{\mathbb{R}\diagdown A}(x)$ satisfy $F_{f}(x)=F_{\widetilde
{f}}(x)=[0,1]$ and at the same time $f(x)\neq\widetilde{f}(x)$ for all
$x\in\mathbb{R}$.
\end{rem}

\begin{definition}
[The map $m(\Phi)$]\label{minmap} Let $\Phi:\mathbb{R}^{d}\rightrightarrows
\mathbb{R}^{\ell}$ be a cusco map. We define the following "minimal" map:
\begin{equation}
m(\Phi)(x):=\bigcap_{N\in\mathcal{N}}\bigcap_{\delta>0}\overline
{\text{\textrm{co}}}\mathrm{\,\,}\Phi(B_{\delta}(x)\mathbf{\diagdown}
N),\quad\text{for all }x\in\mathbb{R}^{d}.\; \label{min}
\end{equation}

\end{definition}

\noindent Thanks to Corollary~\ref{corFi}, we have also
\begin{equation}
m(\Phi)(x)=\bigcap_{\delta>0}\overline{\text{\textrm{co}}}\mathrm{\,\,}
\Phi(B_{\delta}(x)\mathbf{\diagdown}\mathbf{N}_{\Phi}\mathrm{\,\,}).
\label{min-2}
\end{equation}

\begin{prop}
\label{prop_mF}Let $\Phi:\mathbb{R}^{d}\rightrightarrows\mathbb{R}^{\ell}$ be
a cusco map. Then the map $m(\Phi)$ is cusco and satisfies
\begin{align}
m(\Phi)(\bar{x})  &  \subset\bigcap_{\delta>0}\overline{\text{\textrm{co}}
}\mathrm{\,\,}\Phi(B_{\delta}(x))\subset\Phi(\bar{x}),\quad\text{for all }
\bar{x}\in\mathbb{R}^{d}\label{phi1}\\
\;m(\Phi)(\bar{x})  &  =\Phi(\bar{x}),\quad\text{for almost all }\bar{x}
\in\mathbb{R}^{d}. \label{phi2}
\end{align}

\end{prop}

\begin{proof} Fix $N\in\mathcal{N}$, $x\in\mathbb{R}^{d}$ and set
\[
G_{N}(x):=\bigcap_{\delta>0}\overline{\text{co}}\,\Phi(B_{\delta}(x)\!\setminus\!N).
\]
Being a decreasing intersection of nonempty compact convex sets, $G_{N}(x)$ is
itself a nonempty compact convex set. Notice that the family $G_{N}
(x\}_{N\in\mathcal{N}}$ has the finite intersection property. It follows from
(\ref{min}) that the map $m(\Phi)$ has nonempty convex compact values, while
from its definition it follows easily that it is also upper semicontinuous,
that is, $m(\Phi)$ is cusco.\smallskip

We now fix $\varepsilon>0$ and $\bar{x}\in\mathbb{R}^{d}$. Since $\Phi$ is
upper semicontinuous there exists $\delta>0$ such that
\[
\forall x\in B_{\delta}(\bar{x}),\,\Phi(x)\in\Phi(\bar{x})+\varepsilon B.
\]
So $\Phi(B_{\delta}(\bar{x}))\subset\Phi(\bar{x})+\varepsilon B$ and
$\overline{\text{co}}\,\Phi(B_{\delta}(\bar{x}))\subset\Phi(\bar
{x})+2\varepsilon B$ because $\Phi(\bar{x})$ is convex closed. Therefore
\[
\bigcap_{\delta>0}\overline{\text{co}}\,\Phi(B_{\delta}(\bar{x}))\subset
\Phi(\bar{x})+2\varepsilon B.
\]
Taking the intersection over all $\varepsilon>0$ we get
\[
\bigcap_{\delta>0}\overline{\text{co}}\,\Phi(B_{\delta}(\bar{x}))\subset
\bigcap_{\varepsilon>0}(\Phi(\bar{x})+2\varepsilon B)=\Phi(\bar{x}).
\]
This proves (\ref{phi1}). Let us prove (\ref{phi2}). In view of Corollary~\ref{corFi} we get from \eqref{phi1}
\begin{equation}
\forall\bar{x}\in\mathbb{R}^{d},\;m(\Phi)(\bar{x})=\bigcap_{\delta>0}
\overline{\text{co}}\,\Phi(B_{\delta}(x)\diagdown\mathrm{\,\,}N_{\Phi}
)\subset\Phi(\bar{x}). \label{F}
\end{equation}
If $\bar{x}\notin N_{\Phi}$ then
\[
\Phi(\bar{x})\subset\bigcap_{\delta>0}\Phi(B_{\delta}(x)\diagdown
\mathrm{\,\,}N_{\Phi})\subset m(\Phi)(\bar{x}).
\]
Consequently in view of (\ref{F}) we obtain \eqref{phi2} for any $\bar
{x}\notin N_{\Phi}$.
\end{proof}

\section{Characterization of Filippov representable maps}
\label{sec:4}
Let $\mathcal{\hat{C}}(\mathbb{R}^{d},\mathbb{R}^{\ell})$ be the set
of all cusco maps $\Phi:\mathbb{R}^{d}\rightrightarrows\mathbb{R}^{\ell}$. We now define on
$\mathcal{\hat{C}}(\mathbb{R}^{d},\mathbb{R}^{\ell})$  the
equivalence relation
\[
\Phi_{1}\sim\Phi_{2}\quad\Longleftrightarrow\quad\Phi_{1}(x)=\Phi_{2}(x)\text{ for almost all }x\in\mathbb{R}^{d}
\]
and the associated quotient set
\[
\mathcal{\hat{C}}(\mathbb{R}^{d},\mathbb{R}^{\ell})/_{\sim}\;\;:=\{\,[\Phi
],\;\Phi\in\mathcal{\hat{C}}(\mathbb{R}^{d},\mathbb{R}^{\ell})\;\}
\]
where
\[
\lbrack\Phi]:=\{\Psi\in\mathcal{\hat{C}}(\mathbb{R}^{d},\mathbb{R}^{\ell
}),\,\Phi\sim\Psi\;\}.
\]
We also define an order on $\mathcal{\hat{C}}(\mathbb{R}^{d},\mathbb{R}^{\ell
})$ by
\begin{equation}
\Phi_{1}\preceq\Phi_{2}\quad\Longleftrightarrow\quad\Phi_{1}(x)\subseteq
\Phi_{2}(x),\;\text{for all }x\in\mathbb{R}^{d}. \label{eq:order}
\end{equation}

\begin{lem}
[Equivalent elements in $\mathcal{\hat{C}}(\mathbb{R}^{d},\mathbb{R}^{\ell})$
]\label{equiv}For all $\Phi_{1},\Phi_{2}\in\mathcal{\hat{C}}(\mathbb{R}^{d},\mathbb{R}^{\ell})$ we have:
\[
\Phi_{1}\sim\Phi_{2}\quad\Longleftrightarrow\quad m(\Phi_{1})=m(\Phi_{2}).
\]

\end{lem}

\begin{proof} Let $N\in\mathcal{N}$ be such that $\Phi_{1}(x)=\Phi_{2}(x)$
for all $x\in\mathbb{R}^{d}\setminus N$. Fix $\bar{x}\in\mathbb{R}^{d}$. In
view of Corollary~\ref{corFi}, we deduce that for every $\delta>0$
\begin{align*}
\overline{\text{co}}\,\Phi_{1}(B_{\delta}(\bar{x})\setminus
N_{\Phi_{1}})  &  =\overline{\text{co}}\,\Phi_{1}(B_{\delta}(\bar{x}
)\setminus(\mathbf{N}_{\Phi_{1}}\cup\mathbf{N}_{\Phi_{2}}\cup
N)\\
&  =\overline{\text{co}}\,\Phi_{2}\left(  B_{\delta}(\bar{x})\setminus(\mathbf{N}_{\Phi_{1}}\cup\mathbf{N}_{\Phi_{2}}\cup N)\right)
=\overline{\text{co}}\,\Phi_{2}(B_{\delta}(\bar{x})\setminus
\mathbf{N}_{\Phi_{2}})
\end{align*}
because $\Phi_{1}=\Phi_{2}$ on the complement of $N$. By taking intersection
over all $\delta>0$ we obtain
\[
m(\Phi_{1})(\bar{x})=\bigcap_{\delta>0}\overline{\text{co}}\,\Phi
_{1}(B_{\delta}(\bar{x})\setminus\! N_{\Phi_{1}})=\bigcap_{\delta
>0}\overline{\text{co}}\,\Phi_{2}(B_{\delta}(\bar{x})\setminus\!
N_{\Phi_{2}})=m(\Phi_{2})(\bar{x}).
\]
The proof is complete.
\end{proof}

\begin{corol}
[minimality of $m(\Phi)$]Let $\Phi\in\mathcal{\hat{C}}(\mathbb{R}
^{d},\mathbb{R}^{\ell})$. Then $m(\Phi)\in\lbrack\Phi]$ and $m(\Phi)$ is the
minimum element in $[\Phi]$ for the order $\preceq$ defined in
\eqref{eq:order}.
\end{corol}

The fact that every cusco map $\Phi$ is equivalent to $m(\Phi)$
and that the latter is the minimum element of $[\Phi]\ $under set-inclusion,
has an interesting consequence, see \eqref{eq:curious} in the following remark.

\begin{rem}
For every cusco map $\Phi:\mathbb{R}^{d}\rightrightarrows\mathbb{R}^{\ell}$ we
have:
\[
m(\Phi)(x)=\bigcap_{\Phi^{\prime}\sim\Phi}
\Phi^{\prime}(x),\text{\quad for all }x\in\mathbb{R}^{d}\, .
\]
This yields the following relation (which is not completely obvious at a first
glance):
\begin{equation}
\Phi(x)=\bigcap_{\Phi^{\prime}\sim\Phi}
\Phi^{\prime}(x),\quad\text{for a.e. }x\in\mathbb{R}^{d}. \label{eq:curious}
\end{equation}

\end{rem}

We are now ready to establish our main result

\begin{thm}[Characterization of Filippov representable maps]
Let $\Phi:\mathbb{R}^{d}\rightrightarrows\mathbb{R}^{\ell}$ be a cusco map.
Then $\Phi$ is Filippov representable if and only if $\Phi=m(\Phi)$ (that is,
$\Phi$ is the $\preceq$-minimal element in its equivalent class).
\end{thm}

\begin{proof} Let $\Phi:\mathbb{R}^{d}\rightrightarrows\mathbb{R}^{\ell}$ be a Filippov representable cusco map. Then
\[
\Phi(x)=F_{f}(x)=\bigcap_{\delta>0}\overline{\text{\textrm{co}}}
\mathrm{\,\,}\,f(B_{\delta}(x))\setminus\mathbf{N}_{f}),\quad\text{for all
}x\in\mathbb{R}^{d},
\]
where $f:\mathbb{R}^{d}\longrightarrow\mathbb{R}^{\ell}$ is some (bounded)
measurable function. By Lemma~\ref{lem1} we
deduce that  $$f(x)\in\Phi(x) , \; \forall x\in\mathbb{R}^{d}\backslash N_f  .$$ This together with \eqref{min-2} and Lemma~\ref{lem1} yields that for any $x \in \mathbb{R}^{d}$
\[
\Phi(x)=\bigcap_{\delta>0}\,\overline{\text{\textrm{co}}}\mathrm{\,\,}
\,f(B_{\delta}(x)\setminus\left(  \mathbf{N}_{f}\cup\mathbf{N}_{\Phi}\right)
)\subset
\bigcap_{\delta>0}\,\overline{\text{\textrm{co}}}\mathrm{\,\,}
\,\Phi(B_{\delta}(x)\setminus\left(  \mathbf{N}_{f}\cup\mathbf{N}_{\Phi}\right)
).\]
In view of Corollary~\ref{corFi}, we get
\[
\bigcap_{\delta>0}\,\overline{\text{\textrm{co}}}\mathrm{\,\,}
\,\Phi(B_{\delta}(x)\setminus\left(  \mathbf{N}_{f}\cup\mathbf{N}_{\Phi}\right)
) =
\bigcap_{\delta>0}\,\overline{\text{\textrm{co}}}\mathrm{\,\,}
\,\Phi(B_{\delta}(x)\setminus\left( \mathbf{N}_{\Phi}\right)
)
\]
which is equal to $
 m(\Phi)(x)$ by  (\ref{min-2}). This yields
 $\Phi=m(\Phi).$

 To prove the opposite direction, note that by
Theorem \ref{filTh} every cusco map $\Phi$ is equivalent to a Filippov
regularization $F_{f},$ and consequently, $F_{f}=m(F_{f})=m(\Phi)$.
\end{proof}

The following corollary follows directly.

\begin{corol}
The following assertions are equivalent for every cusco map $\Phi
:\mathbb{R}^{d}\rightrightarrows\mathbb{R}^{\ell}$:\smallskip

$\mathbf{(i).}$ $\Phi$ is a Filippov representable map ; \smallskip

$\mathbf{(ii).}$ $\Phi=m(\Phi)$ ;\smallskip

$\mathbf{(iii).}$ for every $\bar{x}\in\mathbb{R}^{d}$ and $N\in\mathcal{N}$ we have:
\[
\overline{\mathrm{co}}\,\left(  \underset{x\notin
N\,x\longrightarrow\bar{x}}{\lim\sup}\Phi(x)\right)  =\Phi(\bar{x})\,.
\]

\end{corol}

\noindent Whenever $\Phi$ is cusco, the left-hand side of (iii) above is
always contained in $\Phi(\bar{x})$. According to (ii) above,
it is very easy to obtain explicit examples of cusco maps that are not
Filippov representable. Indeed, take any measurable function $f,$ consider its
Filippov regularization $F_{f}$ and modify it at some point $\bar{x}$ (or at
all points of a discrete set) to get an equivalent cusco map $\Phi$
different from $F_{f}$. Indeed, it is sufficient to replace $F_{f}(\bar{x})$
by any convex compact strict superset $\Phi(\bar{x})\supset F_{f}(\bar{x})$.
Then $\Phi$ is not Filippov representable, since $\Phi\neq F_{f}
=m(F_{f})=m(\Phi)$, see forthcoming examples.

\bigskip

\begin{example}
\textbf{(i).} We deduce easily that the following cusco maps, based on a
one-point modification of the minimal map $F_{f}(x)=\{0\},$ for all
$x\in\mathbb{R}$ (trivial regularization of the constant function $f\equiv0$),
cannot be obtained as Filippov regularizations:
\[
\Phi_{1}(x)=\left\{
\begin{array}
[c]{cc}
\lbrack0,1], & \text{if }x=0 \smallskip \\
\{0\}, & \text{if }x\neq0
\end{array}
\right.  \text{ \quad and \quad }\Phi_{2}(x)=\left\{
\begin{array}
[c]{cc}
\lbrack-1,1], & \text{if }x=0 \smallskip \\
\{0\}, & \text{if }x\neq 0.
\end{array}
\right.
\]
It is worth noting that $\Phi_{2}$ cannot even be a Krasovskii regularization
of a function, while $\Phi_{1}=K_{g}$, where $g(x)=0,$ for $x\neq0$ and
$g(0)=1.$\smallskip

\noindent\textbf{(ii).} A slightly more elaborated example of a function that can
neither be obtained as Filippov nor as Krasovskii regularization is the
following:
\[
\Phi_{3}(x)=\left\{
\begin{array}
[c]{cc}
\lbrack-\frac{1}{m},\frac{1}{m}], & \text{if }x=p/m\in\mathbb{Q}
\diagdown\{0\} \medskip \\
\{0\}, & \text{if }x\notin\mathbb{Q}\diagdown\{0\}.
\end{array}
\right.
\]
where every nonzero rational number is given its irreducible form $p/m,$ where
$p,m$ are relatively prime integers.\smallskip

\noindent\textbf{(iii).} Let us define the following measurable function:
\[
f(x)=\left\{
\begin{array}
[c]{cc}
\frac{1}{m}, & \text{if }x=p/m\in\mathbb{Q}\diagdown\{0\}\medskip \\
\{0\}, & \text{if }x\notin\mathbb{Q}\diagdown\{0\}.
\end{array}
\right.
\]

Then for every $x\in\mathbb{R}$ we have: $F_{f}(x)=\{0\}$ and $K_{f}
(x)=[0,f(x)].$ In particular $F_{f}\sim K_{f}$ and consequently, the cusco map
$\Phi=K_{f}$ cannot be represented as a Filippov regularization.
\end{example}

\section{Characterization of Clarke subdifferentials}

In this section we deal with the problem of determining whether a cusco map $\Phi\in\mathcal{\hat{C}}(\mathbb{R}^{d},\mathbb{R}^{d})$ is the
Clarke subdifferential of some locally Lipschitz function $\mathbb{\varphi
}:\mathbb{R}^{d}\longrightarrow\mathbb{R}$. A full characterization of such
maps has been given in \cite{BMS1999} and relevant results had been previously
established in \cite{BMW1997}. We shall complement the results of
\cite{BMS1999} by establishing, via our approach, another elegant
characterization of Clarke subdifferentials. Our method is based on the
characterization of Filippov representability (for the case $\ell=d$) together
with a \emph{nonsmooth Poincar\'{e} condition}. This latter has been recently stated
and used independently in \cite{Flores-Master} and \cite{CK2017} for a
different purpose (namely, to identify the free space of a finite-dimensional
Euclidean space). Before we proceed, let us recall the relevant statement.

\begin{thm}
[nonsmooth Poincaré condition (Proposition~3.2(ii) in \cite{CK2017})]\label{Isom} Let $\mathcal{U}\neq \emptyset$ be an open
convex subset of $\mathbb{R}^{d}$. An essentially (locally) bounded measurable
function $f:\mathcal{U}\longrightarrow\mathbb{R}^{d}$ is equal almost
everywhere to the derivative of a (locally) Lipschitz function $\varphi
:\mathcal{U}\longrightarrow\mathbb{R}$ if and only if
\begin{equation}
\partial_{i}f_{j}=\partial_{j}f_{i}\,\;\text{for all }i,j\in\{1,\ldots
,d\},\label{eq:P}
\end{equation}
where $\partial_{i}f_{j}$ denotes the partial derivative (in the sense of
distributions) of the $j$-th component of $f$ with respect to $x_{i}$. That
is, if $\mathcal{C}_{0}^{\infty}(\mathcal{U})$ is the space of compactly
supported $\mathcal{C}^{\infty}$-functions on $\mathcal{U}$ (test functions),
then \eqref{eq:P} becomes:
\[
\int_{\mathcal{U}}f_{j}(x)\frac{\partial\psi}{\partial x_{i}}(x)dx=\int
_{\mathcal{U}}f_{i}(x)\frac{\partial\psi}{\partial x_{j}}(x)dx,\quad\text{for
every }\psi\in\mathcal{C}_{0}^{\infty}(\mathcal{U}).
\]
\end{thm}

We now give an elegant characterization of Clarke
subdifferentials in the spirit of this work.

\begin{thm}
[Characterization of Clarke subdifferentials]\label{th_Clarke}Let
$\Phi:\mathbb{R}^{d}\rightrightarrows\mathbb{R}^{d}$ be a cusco map. The
following are equivalent:\smallskip \newline
$\mathbf{(i).}$ \ $\Phi=\partial\mathbb{\varphi}$ for some
locally Lipschitz function $\mathbb{\varphi}:\mathbb{R}^{d}\longrightarrow
\mathbb{R}$ ;\smallskip\newline
$\mathbf{(ii).}$ $\Phi=F_{f}$ for some measurable selection $f$ of
$\Phi$ that satisfies \eqref{eq:P}.
\end{thm}

\begin{proof} (i)$\Longrightarrow$(ii). Assume that $\Phi=\partial
\mathbb{\varphi}$ for a locally Lipschitz function $\mathbb{\varphi
}:\mathbb{R}^{d}\longrightarrow\mathbb{R}$. Then by Rademacher's theorem,
there exists $N_{\mathbb{\varphi}}\in\mathcal{N}$ such that the derivative
$\nabla\mathbb{\varphi}(x)$ exists for all $x\in\mathbb{R}^{d}\diagdown
N_{\mathbb{\varphi}}.$ For $x\in N_{\mathbb{\varphi}},$ pick $s(x)\in
\partial\mathbb{\varphi}(x)$ and set
\[
f(x)=\left\{
\begin{array}
[c]{cc}%
\nabla\mathbb{\varphi}(x), & \text{if }x\in\mathbb{R}^{d}\diagdown
N_{\mathbb{\varphi}}\\
s(x), & \text{if }x\in N_{\mathbb{\varphi}}.
\end{array}
\right.
\]
Then $f:\mathbb{R}^{d}\longrightarrow\mathbb{R}^{d}$ is a measurable selection
of $\partial\mathbb{\varphi}$ and (being a.e. equal to a gradient) it
satisfies \eqref{eq:P}, see \cite[Proposition 3.1 (ii)]{CK2017}. Moreover,
\begin{equation}
F_{f}(x):=\bigcap_{\delta>0}\,\overline{\text{\textrm{co}}}\mathrm{\,\,}
\,f(B_{\delta}(x)\setminus N_{\varphi})=\bigcap_{\delta>0}\,\overline
{\text{\textrm{co}}}\mathrm{\,\,}\left\{  \nabla\varphi(x^{\prime}):x^{\prime
}\in B_{\delta}(x)\diagdown N_{\varphi}\right\}  .\label{aa1}
\end{equation}
Since $\varphi$ is locally Lipschitz, we deduce (\cite[Chapter~2.6]{Clarke})
\begin{equation}
\bigcap_{\delta>0}\,\overline{\text{\textrm{co}}}\mathrm{\,\,}\left\{
\nabla\varphi(x^{\prime}):x^{\prime}\in B_{\delta}(x)\diagdown N_{\varphi
}\right\}  =\,\overline{\text{\textrm{co}}}\mathrm{\,}\left\{  \lim
_{x_{n}\longrightarrow x}\nabla\varphi(x_{n}):\,\{x_{n}\}\subset\mathbb{R}
^{d}\diagdown N_{\mathbb{\varphi}}\right\}  =\partial\varphi(x),\label{aa2}
\end{equation}
which shows that (ii) holds for $f$ being equal to $\nabla\varphi$
a.e.\medskip

(ii)$\Longrightarrow$(i). Assume that $\Phi=F_{f}$, where $f:$ $\mathbb{R}
^{d}\longrightarrow\mathbb{R}^{d}$ is a measurable selection of $\Phi$ that
satisfies \eqref{eq:P}. Then by Theorem~\ref{Isom}, there
exists a locally Lipschitz function $\mathbb{\varphi}:\mathbb{R}%
^{d}\longrightarrow\mathbb{R}$ such that $f(x)=\nabla\varphi(x),$ for a.e.
$x\in\mathbb{R}^{d}.$ Then it follows from Proposition~\ref{prop_filreg}(v)
and (\ref{aa1}), (\ref{aa2}) above that
\[
\partial\varphi(x)=F_{\nabla\varphi}(x)=F_{f}(x)=\Phi(x) \mbox{ for all } x\in \mathbb{R}^{d}.
\]
\end{proof}

\begin{rem}
(i) It is possible to have  $\Phi=F_{f}$, without $\Phi$ being a
subdifferential; consider for instance the function $f(x_{1},x_{2}
)=(x_{2},-x_{1}),$ for all $(x_{1},x_{2})\in\mathbb{R}^{2}$ (which obviously
fails \eqref{eq:P}). Then $\Phi=f$ cannot be a subdifferential.\medskip
\newline(ii) It is possible to have infinite many measurable selections
$f(x)\in\Phi(x),$ for all $x\in\mathbb{R}^{d}$, each of which satisfies the
nonsmooth Poincar\'{e} condition \eqref{eq:P}. Indeed, if we take $\Phi$ to be
identically equal to the closed ball $\bar{B}$ for all $x\in\mathcal{U},$ then
the set of all measurable selections that satisfy \eqref{eq:P} contains
isometrically the unit ball of the nonseparable Banach space $\ell^{\infty
}(\mathbb{N})$, see \cite{DF2019}.\medskip\newline(iii) If $\Phi=F_{f}$ and
$f$ is unique a.e. and satisfies \eqref{eq:P}, then  by
Theorem~\ref{th_Clarke}, $\Phi=\partial\varphi$ and $f=\nabla\varphi$ a.e. It
follows that the locally Lipschitz function $\varphi$ is unique up to a constant.
\end{rem}

\vspace{0.8cm}

\textbf{Acknowledgments.} A major part of this work was done during a research
visit of the first author to the University of Brest (2019) and of
the third author to the University of Chile (January 2020). These authors wish
to thank their hosts for hospitality and acknowledge financial support from
the Laboratoire de Math\'{e}matiques de Bretagne Atlantique and from FONDECYT
grant 1171854 respectively. The second author's research has been supported by
the grants CMM-AFB170001, FONDECYT 1171854 (Chile) and PGC2018-097960-B-C22,
MICINN (Spain) and ERDF (EU). The research of the first and the third author has been supported by the Air Force Office of Scientific Research under award number FA9550-18-1-0254.

\bigskip

\noindent Mira BIVAS\medskip

\noindent Faculty of Mathematics and Informatics, Sofia University\newline%
\noindent James Bourchier Boul. 5, 1126 Sofia, Bulgaria
\\and\\
Institute of Mathematics and Informatics, Bulgarian Academy of Sciences\\
G.Bonchev str., bl. 8, 1113 Sofia, Bulgaria

\smallskip

\noindent E-mail: \texttt{mira.bivas@math.bas.bg}\thinspace\ \thinspace
\ \newline\noindent \texttt{https://www.researchgate.net/profile/Mira\_Bivas}
\smallskip

\noindent Research supported by the Air Force Office of Scientific Research (FA9550-18-1-0254)

\bigskip

\noindent Aris DANIILIDIS\medskip

\noindent DIM--CMM, UMI CNRS 2807\newline Beauchef 851, FCFM, Universidad de
Chile \smallskip

\noindent E-mail: \texttt{arisd@dim.uchile.cl} \newline
\texttt{http://www.dim.uchile.cl/~arisd/}
\smallskip

\noindent Research supported by the grants: \newline CMM AFB170001, FONDECYT
1171854 (Chile), PGC2018-097960-B-C22 (Spain and EU).

\vspace{0.5cm}

\noindent Marc QUINCAMPOIX\medskip

\noindent Laboratoire de Math\'{e}matiques de Bretagne Atlantique (CNRS UMR
6205)\newline\noindent Univ  Brest, 6, avenue Victor Le Gorgen,
29200 Brest, France

\medskip

\noindent E-mail: \texttt{marc.quincampoix@univ-brest.fr}\newline%
\noindent\texttt{http://mespages.univ-brest.fr/\symbol{126}quincamp/}

\medskip

\noindent Research supported by the Air Force Office of Scientific Research (FA9550-18-1-0254)
\end{document}